\documentclass[10pt]{amsart}
\usepackage{amssymb}
\usepackage{graphicx}
\usepackage[all,cmtip]{xy}
\newcommand{\D}{\mathbb{D}}

\renewcommand{\P}{\mathbb{P}}

\newcommand{\SC}{{\mathcal{C}}}

\newcommand{\SL}{{\mathcal{L}}}

\newcommand{\SN}{{\mathcal{N}}}

\renewcommand{\SS}{{\mathcal{S}}}

\newcommand{\Z}{\mathbb{Z}}

\newcommand{\R}{\mathbb{R}}

\renewcommand{\S}{\mathbb{S}}
\renewcommand{\SC}{\mathcal{C}}

\newcommand{\Diff}{\operatorname{Diff}}

\newtheorem{proposition}{Proposition}
\newtheorem{theorem}[proposition]{Theorem}
\newtheorem{definition}[proposition]{Definition}
\newtheorem{lemma}[proposition]{Lemma}

\newtheorem{remark}[proposition]{Remark}

\theoremstyle{definition}
\newtheorem{example}[proposition]{Example}

\newtheorem{innercustomthm}{Theorem}
\newenvironment{customthm}[1]
  {\renewcommand\theinnercustomthm{#1}\innercustomthm}
  {\endinnercustomthm}

\begin{document}
\title{Overtwisted disks and Exotic Symplectic Structures}
\date{January, 2014}

\author{Roger Casals}
\address{Instituto de Ciencias Matem\'aticas CSIC--UAM--UCM--UC3M,
C. Nicol\'as Cabrera, 13--15, 28049, Madrid, Spain.}
\email{casals.roger@icmat.es}

\begin{abstract}
The symplectization of an overtwisted contact $(\R^3,\xi_{ot})$ is shown to be an exotic symplectic $\R^4$. The technique can be extended to produce exotic symplectic $\R^{2n}$ using a GPS--structure and applies to symplectizations of appropriate open contact manifolds.
\end{abstract}
\maketitle

\noindent Let $(\R^{2n},\omega_0)$ be the standard symplectic structure on $\R^{2n}$. A symplectic structure $\omega$ on $\R^{2n}$ is exotic if there exists no symplectic embedding
$$\varphi:(\R^{2n},\omega)\longrightarrow(\R^{2n},\omega_0).$$
The non--existence of embedded exact Lagrangians in $(\R^{2n},\omega_0)$ and the $h$--principle for immersions imply that $\R^{2n}$ admits an exotic symplectic structure for $n\geq2$. See Exercise b. in page 344 in \cite{Gr}.\\

A symplectic structure on $\R^2$ is symplectomorphic to the standard symplectic structure. In the case of $\R^4$ and $\R^6$ exotic symplectic structures are provided in \cite{BP} and \cite{Mu} respectively. The articles \cite{ML,SS} contain an approach to exotic Stein structures. Note that a finite type Stein manifold diffeomorphic to $\R^4$ has to be symplectomorphic to $(\R^4,\omega_0)$. The detection of exotic symplectic structures often relies on symplectic arguments, such as the study of embedded Lagrangians. See also \cite{Zu}.\\

\noindent The aim of the present article is to show that techniques in contact topology can also be used to construct and detect exotic symplectic structures. In particular the exotic symplectic structures we describe are simple and explicit. The arguments we provide use known obstructions to fillability. See \cite{Ni,NP}. The proofs in these articles require pseudo--holomorphic curves. This is the only place where non--elementary contact topology is invoked. The main result is the following

\begin{theorem}\label{thm:main}
Let $(\R^3,\xi_{ot})$ be an overtwisted contact structure, then the symplectization $\SS(\R^3,\xi_{ot})$ endows $\R^4$ with an exotic symplectic structure.
\end{theorem}

\begin{example}\label{ex:ot}
Let $(\rho,\varphi,z)\in\R^3$ be cylindrical coordinates and $(\R^3,\xi_1)$ the contact structure defined by the kernel of the contact form
$$\alpha_1=\cos\rho dz+\rho\sin\rho d\varphi.$$
Consider the symplectic $2$--form $\omega_1=d(e^t\alpha_1)$ on $\R^4\cong\R^3\times\R(t)$. Then $(\R^4,\omega_1)$ is an exotic symplectic structure.\hfill$\Box$
\end{example}
The arguments we use apply to several open contact manifolds. For instance:
\begin{theorem}\label{thm:openmfld}
Let $(M,\xi)$ be an exact symplectically fillable contact 3--fold and $(U,\xi)\subset (M,\xi)$ an open contact submanifold. Consider an overtwisted contact structure $(U,\xi_ {ot})$. Then $\SS(U,\xi)$ is not symplectomorphic to $\SS(U,\xi_{ot})$.
\end{theorem}
The same techniques allow us to prove similar results in higher--dimensions. In particular we prove that the exotic symplectic structures obtained in Theorem \ref{thm:main} are stable.
\begin{theorem}\label{thm:stable}
Let $(\R^3,\xi_{ot})$ be an overtwisted contact structure and $(\SS(\R^3,\xi_{ot}),\omega_{ot})$ its symplectization. Then $(\SS(\R^3,\xi_{ot})\times\R^{2n-4},\omega_{ot}+\omega_0)$ endows $\R^{2n}$ with an exotic symplectic structure.
\end{theorem}
The appropriate analogue of Theorem \ref{thm:openmfld} also holds for higher dimensions.\\

I have been informed that Y. Chekanov may have a different argument for Theorem \ref{thm:main}. K. Niederkr\"uger explained to me that one can use bLobs as a generalization for the GPS--structure. The paper is organized as follows. Sections \ref{sec:pre} and \ref{sec:ot} introduce the ingredients used to prove the above results. The proof of Theorem \ref{thm:main} is detailed in Section \ref{sec:sympOT}.
Section \ref{sec:EX} contains the proof of Theorems \ref{thm:openmfld} and \ref{thm:stable} .\\

{\bf Acknowledgements.} This article stems from a discussion with S. Courte and E. Giroux at ENS Lyon. I am very grateful for their hospitality. I also thank K. Niederkr\"uger and F. Presas for useful discussions. The present work is part of the author activities within CAST, a Research Network Program of the European Science Foundation.

\section{Preliminaries}\label{sec:pre}

\subsection{Contact structures on $\R^3$}\label{ssec:contR3}

The study of contact structures in $\R^3$ yielded to foundational work in contact topology. The first step towards an isomorphism classification was the distinction between the standard contact structure on $\R^3$ and the overtwisted contact structure described in Example \ref{ex:ot}. This is the work of D. Bennequin in \cite{Be}. The isomorphism classification of contact structures on $\R^3$ is completed after the seminal work of Y. Eliashberg in \cite{E1,E2,E3}.\\

\noindent The standard contact structure $\xi_0$ on $\R^{3}(\rho,\varphi,z)$ is defined as the kernel of the contact form
$$\alpha_0=dz+\rho^2d\varphi.$$
This is a normal form of any contact 1--form in a sufficiently small neighborhood of a point in a contact 3--fold.\\

The contact structure $\xi_1$ induced by the contact form
$\alpha_1=\cos\rho dz+\rho\sin\rho d\varphi$
contains an overtwisted disk $\Delta=\{(\rho,\varphi,z):\rho\leq\pi,z=0\}$.
 The arguments in \cite{Be} imply that $(\R^3,\xi_0)$ and $(\R^3,\xi_1)$ are not contactomorphic.\\

Consider the $3$--sphere $\S^3$. The main result in \cite{E1} implies the existence of a unique overtwisted contact structure in each homotopy class of plane distribution on $\S^3$. There are $H^3(\S^3,\pi_3(\R\P^2))=\Z$ homotopy classes. Denote by $\zeta_k$ the overtwisted contact structure in the homotopy class identified with $k\in\Z$. Then $\zeta_k$ restricted to $\S^3\setminus\{p\}$, $p\in\S^3$, defines an overtwisted contact structure on $\R^3$. It will still be denoted $\zeta_k$. The classification result in \cite{E3} is the following

\begin{theorem}
Each contact structure on $\R^3$ is isotopic to one of the structures $\xi_0$, $\xi_1$ or $\zeta_k$, for $k\in\Z$. These structures are pairwise non--contactomorphic.
\end{theorem}

Thus the overtwisted disk $\Delta\subset(\R^3,\xi_1)$ is the local model in a neighborhood of any overtwisted disk. That is, any small ball containing an overtwisted disk in a contact 3--fold is necessarily contactomorphic to $(\R^3,\alpha_1)$.\\

The symplectic structures we consider in this article are constructed with a contact structure. The procedures we use to obtain a symplectic manifold from a contact manifold and viceversa will be explained in the following subsections. This material can be found in \cite{AG}.

\subsection{Symplectization}\label{ssec:symp} Let $(M,\xi)$ be a contact manifold and $\SS(M,\xi)$ be the subbundle of the cotangent bundle $\pi:T^*M\longrightarrow M$ whose fibre at a point $p\in M$ consists of all non--zero linear functions on the tangent space $T_pM$ which vanish on the contact hyperplane $\xi_p\subset T_pM$ and define its given coorientation. Giving $\SS(M,\xi)$ as a subbundle of the cotangent bundle $T^*M$ is tantamount to endowing $M$ with a contact structure.\\

Consider the Liouville 1--form $\lambda$ on $T^*M$, the 2--form $d\lambda$ restricts to a symplectic structure on $\SS(M,\xi)$.

\begin{definition}\label{def:symp}
The symplectization of $(M,\xi)$ is the exact symplectic manifold
$$(\SS(M,\xi),d\lambda|_{\SS(M,\xi)}).$$
\end{definition}

\noindent In our perspective the primitive is not part of the data, only the symplectic structure is. In the study of Liouville domains the primitive is also part of the structure of a symplectization. This is not the case. The bundle $\pi:\SS(M,\xi)\longrightarrow M$ is a trivial principal $\R^+$--bundle. The sections of $\pi$ are contact forms for the contact structure $\xi$. A choice of contact form $\alpha$ defines a trivialization $\SS(M,\xi)\cong M\times\R^+(t)$. In terms of this splitting $\lambda|_{\SS(M,\xi)}=t\alpha$. In case a contact form $\alpha$ has been given to $(M,\xi)$, its symplectization $\SS(M,\xi)$ will also be denoted by $\SS(M,\alpha)$. Contactomorphic contact manifolds yield symplectomorphic symplectizations.\\

\noindent In this article $\R^{2n+2}$ is identified with the total space of $\SS(\R^{2n+1},\alpha)$. This is done with the diffeomorphism $e:\R(t)\longrightarrow\R^+(t)$, $e(t)=e^t$. The use of $t\in\R^+$ is more convenient since we consider $t$ to be a radius in certain polar coordinates of an annulus. The coordinate $e^t\in\R^+$ shall sometimes be used, as in the following example.\\

\begin{example}\label{ex}
Consider $\R^{2n+1}$ with coordinates $(x_1,y_1,\ldots,x_n,y_n,z)=(\rho_1,\varphi_1,\ldots,\rho_n,\varphi_n,z)$ and endowed with the contact form
$$\alpha_0=dz+\sum_{i=1}^n \rho_i^2d\varphi_i.$$
Its symplectization is the symplectic manifold $(\R^{2n+1}\times\R(t),d(e^t\alpha_0))$. This is symplectomorphic to the standard symplectic $(\R^{2n+2},\omega_0)$ where $\omega_0=\sum_{i=1}^n dx_i\wedge dy_i+dt\wedge dz$. Indeed, consider the contact form
$\widetilde{\alpha}_0=dz-\sum_{i=1}^n y_i\cdot dx_i$ on $\R^{2n+1}$. It is readily seen that $(\R^{2n+1},\ker\alpha_0)\cong(\R^{2n+1},\ker\widetilde\alpha_0)$.
Then the diffeomorphism
$$f:\R^{2n+2}\longrightarrow\R^{2n+2},\quad f(x_1,y_1,\ldots,x_n,y_n,z,t)=(x_1,e^ty_1,\ldots,x_n,e^ty_n,e^tz,t)$$
pulls--back the standard symplectic form to
$$f^*\left(\sum_{i=1}^n dx_i\wedge dy_i+dt\wedge dz\right)=\sum_{i=1}^n \left(e^t dx_i\wedge dy_i+e^ty_i\cdot dx_i\wedge dt\right)+e^t dt\wedge dz=$$
$$=d\left(e^tdz-\sum_{i=1}^n e^ty_i\cdot dx_i)\right)=d(e^t\widetilde\alpha_0).$$
Hence $\SS(\R^{2n+1},\alpha_0)\cong(\R^{2n+2},\omega_0)$. The permutation in the variables $(z,t)$ has its geometric origin in the dichotomy between convexity and concavity. Confer Section \ref{ssec:diskCS}.
\end{example}

\begin{remark}\label{rmk:sphere}
The contact structure $\xi_0=\ker\alpha_0$ on $\R^{2n+1}$ extends to a contact structure $(\S^{2n+1},\xi_0)$ via the one point compactification.
\end{remark}

\noindent It is a natural question whether $\SS(\R^3,\alpha_0)$ and $\SS(\R^3,\alpha_1)$ are symplectomorphic. A symplectic topology proof could be finding exact Lagrangian tori in $\SS(\R^3,\alpha_1)$, since these do not exist in $\SS(\R^3,\alpha_0)$. Such a Lagrangian tori would also distinguish the symplectomorphism type of $\SS(\R^3,\alpha_0)$ and $\SS(\R^3,\zeta_k)$, $k\in\Z$. Instead, we shall use contact topology.\\

Note also that the classic symplectic invariants such as volume, width and symplectic capacities are necessarily infinite in the symplectization of a contact manifold.

\subsection{Contactization}\label{ssec:cont} Let $(V,\lambda)$ be an exact symplectic manifold with a Liouville 1--form $\lambda$.
\begin{definition}\label{def:cont}
The contactization $\SC(V,\lambda)$ of $(V,\lambda)$ is the contact manifold $(V\times\R(s),\lambda-ds)$.
\end{definition}
\noindent Note that a different choice of primitive $\lambda$ for the symplectic structure $d\lambda$ on $V$ may lead to a different contact structure on $V\times\R$. In case there exists a function $f:V\longrightarrow\R$ such that $\lambda_0-\lambda_1=df$, the map
$$F:\SC(V,\lambda_0)\longrightarrow\SC(V,\lambda_1),\quad F(p,s)=(p,s-f(p))$$
is a strict contactomorphism. Note that for $V=\R^{2n}$, or more generally $H^1(V;\R)=0$, such a potential $f$ exists.\\

\noindent The coordinate $s\in\R$ in $V\times\R(s)$ can be considered to be an angle $s\in\S^1$. In particular, the contactization $\SC(V,\lambda)$ can be compactified to $V\times\S^1(s)$. This compactification is also referred to as the contactization of $(V,\lambda)$.

\subsection{Contact fibration of $\SC\SS(M,\xi)$ over $\D^2$}\label{ssec:diskCS} Let $(M,\xi)$ be a contact manifold and $\alpha$ an associated contact form. The symplectization $\SS(M,\alpha)\cong(M\times\R^+(t),d(t\alpha))$ is an exact symplectic manifold. Thus $\SC\SS(M,\xi)$ is defined, the choice of Liouville form in this case is $\lambda=t\alpha$. The underlying smooth manifold $M\times\R^+(t)\times\R(s)$ can be compactified to $M\times\R^+(t)\times\S^1(s)$. Then the coordinates $(t,s)$ can be considered to be polar coordinates on $\R^2\setminus\{0\}$ and projection onto the latter two factors defines a smooth fibration
$$\pi:M\times\R^+(t)\times\S^1(s)\longrightarrow\R^2\setminus\{0\}.$$
A smooth fibration $p:X\longrightarrow B$ is said to be contact for a codimension--1 distribution $\xi\subset TX$ if $\xi$ restricts to a contact structure on any fibre. The map $p=\pi$ satisfies this condition for the natural contact structure on $\SC\SS(M,\xi)$.

\begin{proposition}\label{prop:contfib}
The smooth fibre bundle
$$\pi:M\times\R^+\times\S^1\longrightarrow\R^2\setminus\{0\},\quad (p,t,s)\longmapsto (t,s).$$
is a contact fibration for $\xi=\ker\{t\alpha-ds\}$. There exists a diffeomorphism $G$ between contact fibrations such that
$$\xymatrix{
(M\times\R^+\times(0,2\pi),\ker\{\alpha+r^2d\theta\})\ar@{->}[r]^{G} \ar@{->}[d]_p & (M\times\R^+\times\R,\ker\{t\alpha-ds\})\ar@{->}[d]_p\\
\R^+\times(0,2\pi)\ar@{->}[r]^{\cong} & \R^+\times\R
}$$
is commutative, the map $p$ being in both cases the projection onto the rightmost two factors.
\end{proposition}
\begin{proof}
The first statement is readily verified. For the second statement, consider the following change of coordinates
$$(r,\theta)\in\R^+\times(0,2\pi)\stackrel{g}{\longrightarrow}(t,s)\in\R^+\times\R,\quad 1/t=-4\cos^2\left(\theta/4\right)\cdot r^2,\quad s=\tan\left(\theta/4\right).$$
The map $g$ defines a contactomorphism
$$G=(id,g):(M\times\R^+\times(0,2\pi),\ker\{\alpha+r^2d\theta\})\longrightarrow(M\times\R^+\times\R,\ker\{t\alpha-ds\})$$
since $G^*(\alpha-(1/t)\cdot ds)=\alpha+r^2d\theta$. The map $G$ commutes with the projections.
\end{proof}

From the viewpoint of differential topology the projection $\pi$ from $M\times\R^+\times\S^1$ is appropriate. Nevertheless from a symplectic perspective the two ends $M_-=M\times\{0\}\times\{s_0\}$ and $M_+=M\times\{\infty\}\times\{s_0\}$ are quite different, for any fixed $s_0\in\S^1(s)$. The negative end $M_-$ of a symplectization is concave and the positive end $M_+$ is convex. Consider polar coordinates $(r,\theta)\in\R^2$ restricting to
$$(r,\theta)\in\R^+\times(0,2\pi)=\R^2\setminus L,$$
where $L=\{(r,\theta):r\geq0,\quad \theta=0\}$. Then the convexity of the boundary at infinity leads to the change of coordinates in Proposition \ref{prop:contfib}. This is a more natural symplectic coordinate system: the binding of the natural open book in $\SC\SS(M,\xi)$ induced by polar coordinates on the disk $\D^2(r,\theta)$ lies above the origin of the disk. It is then natural to compactify not only smoothly, but in a contact sense, the contact manifold $(M\times\R^+\times(0,2\pi),\ker\{\alpha+r^2d\theta\})$ to the contact manifold $(M\times\D^2,\ker\{\alpha+r^2d\theta\})$.
\section{Overtwisted disks and GPS}\label{sec:ot}
\noindent The concepts and results of this Section are part of the content of \cite{Ni,NP}.
\begin{definition}\label{def:gps}
Let $(M^5,\xi)$ be a contact $5$--fold and $\xi=\ker\alpha$. A GPS--structure is an immersion $\iota:\S^1\times\D^2(r,\theta)\longrightarrow M$ conforming the following properties
\begin{itemize}
\item[-] $\iota^*\alpha=f(r)d\theta$, for $f\geq0$ and $f(r)=0$ only at $r=0,1$.
\item[-] There exists $\varepsilon>0$ such that the self--intersection points are of the form
$$p_1=(s_1,r_1,\theta)\mbox{ and }p_2=(s_2,r_2,\theta),\quad r_1,r_2\in(\varepsilon,1-\varepsilon).$$
\item[-] There exists an open set with no self--intersection points joining $\S^1\times\{0\}$ and $\S^1\times\partial\D^2$.
\end{itemize}
\end{definition}

\noindent The existence of a GPS--structure partially restricts the fillability properties of the contact manifold $(M,\ker\alpha)$. In particular, we can use the main result in \cite{NP}. It implies the following
\begin{theorem}\label{thm:np} Let $(M,\ker\alpha)$ be a contact manifold with a GPS--structure. Then $(M,\ker\alpha)$ does not admit an exact symplectic filling.
\end{theorem}
\noindent The construction of a GPS--structure through the use of a contact fibration was introduced in \cite{Pr}. In Section 4 of \cite{NP} details for the following result are provided.
\begin{proposition}\label{pro:GPS}
Let $(\R^3,\ker\alpha_{ot})$ be an overtwisted contact structure and $(p,r,\theta)\in\R^3\times\D^2(r_0)$ polar coordinates. There exists $R\in\R^+$ sufficiently large such that the contact manifold $(\R^3\times\D^2(R),\ker\{\alpha_{ot}+r^2d\theta\})$ contains a GPS--structure.
\end{proposition}
\section{Symplectization of an overtwisted structure}\label{sec:sympOT}
\noindent In this section we prove Theorem \ref{thm:main}. Let $(\R^3,\ker\alpha_{ot})$ be an overtwisted contact structure. The idea is simple: the contactization $\SC(\R^4,e^t\alpha_{ot})$ of the exact symplectic manifold $(\R^4,d(e^t\alpha_{ot}))$ is not contactomorphic to $(\R^5,\xi_0)\cong\SC\SS(\R^3,\xi_0)$. Indeed, it will be proven that $\SC(\R^4,e^t\alpha_{ot})$ does not embed into $(\S^5,\xi_0)$ whereas $(\R^5,\xi_0)$ does. The geometric model is that of Subsection \ref{ssec:diskCS} and thus $(\R^4,e^t\alpha_{ot})$ is seen as $(\R^3\times\R^+,t\alpha_{ot})$.

\begin{lemma}\label{lem:m}
Let $(p,r,\theta)\in\R^3\times\R^2$ be polar coordinates and $L=\{(p,r,\theta):r\geq0,\theta=0\}$. There exists a contactomorphism
$$\Phi:(\R^3\times(\R^2\setminus L),\ker\{\alpha_{ot}+r^2d\theta\})\longrightarrow\SC(\R^3\times\R^+,t\alpha_{ot}).$$
\end{lemma}
\begin{proof}
Consider the map $G$ in the proof of Proposition \ref{prop:contfib}. The contactization $\SC(\R^3\times\R^+,t\alpha_{ot})$ is contactomorphic to
$$\SC(\R^3\times\R^+,t\alpha_{ot})=\left(\R^3\times\R^+\times\R(s),\ker\{\alpha_{ot}-(1/t)ds\}\right)\stackrel{G^{-1}}{\longrightarrow}(\R^3\times\R^+\times(0,2\pi),\ker\{\alpha_{ot}+r^2d\theta\})$$
which is $(\R^3\times(\R^2\setminus L),\ker\{\alpha_{ot}+r^2d\theta\})$.
\end{proof}

\begin{lemma}\label{lem:s}
Let $(p,x,y)=(p,r,\theta)\in\R^3\times\R^2$ be cartesian and polar coordinates. There exists a strict contactomorphism
$$\Psi:(\R^3\times\R^2,\ker\{\alpha_{ot}+r^2d\theta\})\longrightarrow(\R^3\times\R^2,\ker\{\alpha_{ot}-ydx\}).$$
which preserves the fibres of the projection onto the second factor.
\end{lemma}

\begin{proof}
The contact form $\alpha_{ot}+r^2d\theta$ in Cartesian coordinates reads $\beta_0=\alpha_{ot}+\frac{1}{2}(xdy-ydx)$. Consider the homotopy of contact forms
$$\beta_t=\alpha_{ot}-ydx+\frac{1-t}{2}(xdy+ydx),\quad t\in[0,1].$$
It begins at $\beta_0$ and ends at $\beta_1=\alpha_{ot}-ydx$. Let us find an isotopy $\tau_t$ solving the equation $\tau_t^*\beta_t=0$. Suppose that $\tau_t$ is the $t$--time flow of a vector field $X_t$. The derivative of the equation reads
$$\tau_t^*(\SL_{X_t}\beta_t+\dot{\beta}_t)=0,\mbox{ i.e. }d\iota_{X_t}\alpha_t+\iota_{X_t}d\alpha_t-\frac{1}{2}(xdy+ydx)=0.$$
Note that $d(xy)=xdy+ydx$ and thus the autonomous vector field
$$X=\frac{xy}{2}R_{ot}$$
is the solution to this equation, where $R_{ot}$ denotes the Reeb vector field of $\alpha_{ot}$. The vector field $X$ is a complete vector field in $\R^3\times\R^2$. Let $\tau$ be its 1--time flow. Then the diffeomorphism
$$\Psi:\R^3\times\R^2\longrightarrow\R^3\times\R^2,\quad\Psi(p,x,y)=(\tau(p,x,y),x,y)$$
satisfies $\Psi^*(\alpha_{ot}-ydx)=\alpha_{ot}+r^2d\theta$.
\end{proof}

\begin{theorem}\label{thm:otnotinstd}
There exists no contact embedding $\SC\SS(\R^3,\alpha_{ot})\longrightarrow\SC\SS(\R^3,\alpha_0)$.
\end{theorem}
\begin{proof}
The contact manifold $\SC\SS(\R^3,\alpha_0)$ is contactomorphic to $(\R^5,\ker\alpha_0)$. Thus it embeds via the inclusion $j$ into $(\S^5,\ker\alpha_0)$. The contact manifold $(\S^5,\ker\alpha_0)$ admits an exact symplectic filling by the standard symplectic ball $(\D^6,\omega_0|_{\D^6})$. Suppose that there exists a contact embedding
$$h:\SC\SS(\R^3,\alpha_{ot})\longrightarrow\SC\SS(\R^3,\alpha_0).$$
Proposition \ref{pro:GPS} implies the existence of a GPS--structure on the contact manifold
$$\SN_R=(\R^3\times\D^2(R),\ker\{\alpha_{ot}+r^2d\theta\})$$
for $R$ large enough. Let us show that $\SN_R$ contact embeds into $\SC\SS(\R^3,\alpha_{ot})$. Lemma \ref{lem:m} identifies this contactization via $\Phi$ with $(\R^3\times(\R^2\setminus L),\ker\{\alpha_{ot}+r^2d\theta\})$. The contactomorphism $\Psi$ in Lemma \ref{lem:s} allows us to use $(\R^3\times(\R^2\setminus L),\ker\{\alpha_{ot}-ydx\})$.\\

Consider an arbitrary $R_0\in\R^+$, the inclusion $i:\D^2(R_0)\longrightarrow\R^2(x,y)$ as a disk centered at the origin and the diffeomorphism $f_R\in\Diff(\R^2)$ defined as $f_R(x,y)=(x-R,y)$. The image of $\SN_{R_0}$ via the contact embedding $(id,f_R\circ i)$ is contained in $\R^3\times(\R^2\setminus L)$ if $R>R_0$. It is readily verified that
$$\gamma:\Phi\circ\Psi^{-1}\circ (id,f_{2R_0}\circ i):\SN_{R_0}\longrightarrow\SC\SS(\R^3,\alpha_{ot})$$
is a contact embedding. The radius $R_0$ can be chosen arbitrarily large. The map $j\circ h\circ \gamma$ endows $(\S^5,\ker\alpha_0)$ with a GPS--structure. This contradicts Theorem \ref{thm:np}.
\end{proof}


\noindent{\it Proof of Theorem \ref{thm:main}}: Suppose that symplectic structure $\SS(\R^3,\alpha_{ot})$ is not exotic. Then there exists an embedding $i:\SS(\R^3,\alpha_{ot})\longrightarrow \SS(\R^3,\alpha_0)$. It induces a contact embedding
$$j:\SC\SS(\R^3,\alpha_{ot})\longrightarrow\SC\SS(\R^3,\alpha_0).$$
This contradicts Theorem \ref{thm:otnotinstd}.\hfill$\Box$\\

Note that the symplectic structure $\SS(\R^3,\xi_{ot})$ is never standard at infinity. It has been proven by M. Gromov that a symplectic structure on $\R^4$ standard at infinity is necessarily isomorphic to the standard symplectic structure $(\R^4,\omega_0)$.\\

\noindent The contact structures $\xi_0$ and $\xi_1$ on $\R^3$ are homotopic through contact structures. This homotopy can be obtained by dilating the overtwisted disks off to infinity. This geometric path of contact structures yields a path of exact symplectic forms joining the standard symplectic structure $\omega_0$ and the symplectic structure on $\SS(\R^3,\xi_1)$. A visual homotopy between $\xi_0$ and $\zeta_k$ can be readily constructed using contractions to a Darboux ball. This also induces a homotopy between $\omega_0$ and the symplectic form of $\SS(\R^3,\zeta_k)$.
\section{Examples of Non--Isomorphic Symplectizations}\label{sec:EX}
In this Section we provide details on Theorem \ref{thm:openmfld} and Theorem \ref{thm:stable}.
\subsection{Open contact 3--folds}\label{ssec:open}
\noindent In Section \ref{sec:ot} we have shown that $\SS(\R^3,\alpha_0)$ is not symplectomorphic to $\SS(\R^3,\alpha_{ot})$. The procedure we used yields several examples of open manifolds exhibiting this behaviour. In particular Theorem \ref{thm:openmfld} stated in the introduction.
\begin{customthm}{3}\label{thm:openmfld}{\it
Let $(M,\xi)$ be an exact symplectically fillable contact manifold and $(U,\xi)\subset (M,\xi)$ an open contact submanifold. Consider an overtwisted contact structure $(U,\xi_ {ot})$. Then $\SS(U,\xi)$ is not symplectomorphic to $\SS(U,\xi_{ot})$.}
\end{customthm}
\begin{proof}
Consider an exact symplectic filling $(W,\lambda)$ for $(M,\xi)$, $\xi=\ker\alpha$. Note that $\SS(M,\xi)$ embeds into $(W,\lambda)$ as a neighborhood of the boundary. The contact 5--fold $\SC(W,\lambda)=(W\times\S^1,\lambda-ds)$ has boundary $M\times\S^1$. In order to obtain a closed contact 5--fold $(X,\Xi)$ we glue the manifold $(M\times\D^2,\alpha+\rho^2d\varphi)$ along their common boundary $M\times\S^1$. The manifold $(X,\Xi)$ admits an exact symplectic filling.\\

Observe that the open contact manifold $(U,\xi)$ embeds into $(X,\Xi)$ with an arbitrarily large neighborhood. Indeed, $(M,\xi)$ has an arbitrarily large symplectic neighborhood in $(W,\lambda)$. For instance, it can be obtained by expanding a given neighborhood with the Liouville flow.\\

The open contact manifold $\SC\SS(U,\xi_{ot})$ contains a GPS--structure. Suppose that $\SS(U,\xi)$ is symplectomorphic to $\SS(U,\xi_{ot})$, then $\SS(U,\xi_{ot})$ embeds into $(W,\lambda)$. Hence the contact manifold $\SC\SS(U,\xi_{ot})$ embeds into $(X,\Xi)$. This contradicts Theorem \ref{thm:np}.
\end{proof}

\begin{remark}
The manifold $(X,\Xi)$ used in the proof is not unique. The relative suspension using a composition of positive Dehn twists also yields an exact symplectically fillable manifold and the argument applies.
\end{remark}

\subsection{Higher Dimensions}\label{ssec:hd}
Consider an overtwisted contact structure $(\R^3,\xi_{ot})$ and polar coordinates $(\rho_1,\varphi_1,\ldots,\rho_{n-2},\varphi_{n-2})\in\R^{2n-4}$. The contact structure $\xi_{ex}$ defined by the kernel of the $1$--form
$$\alpha_{ex}=\alpha_{ot}+\sum_{i=1}^{n-2}\rho_i^2d\varphi_i$$
contains a GPS--structure. Thus it is not contactomorphic to $(\R^{2n-1},\xi_0)$. The statement of Proposition \ref{pro:GPS} also holds for the contact manifold $(\R^{2n-1},\xi_{ex})$. That is, there exists a GPS--structure on $(\R^{2n-1}\times\D^2(R),\alpha_{ex}+r^2d\theta)$. Confer \cite{NP} for details. The existence of this GPS--structure and the analogues of Lemmas \ref{lem:m} and \ref{lem:s} prove that $\SC\SS(\R^{2n-1},\alpha_{ex})$ does not contact embed into $\SC\SS(\R^{2n-1},\alpha_0)$. The same argument used in Theorem \ref{thm:main} yields the following
\begin{proposition}\label{prop:higher}
Let $(\R^3,\xi_{ot})$ be an overtwisted contact structure, $\xi_{ot}=\ker\alpha_{ot}$. Then the symplectization $\SS(\R^{2n-1},\alpha_{ex})$ endows $\R^{2n}$ with an exotic symplectic structure.\hfill$\Box$
\end{proposition}
This allows us to conclude Theorem \ref{thm:stable} stated in the introduction.\\

{\it Proof of Theorem \ref{thm:stable}}: Consider the diffeomorphism
$$f:\R^{2n}\longrightarrow\R^{2n},\quad f(\rho,\varphi,z,t;\rho_1,\varphi_1,\ldots,\rho_{n-2},\varphi_{n-2})=(\rho,\varphi,z,t;e^{t/2}\rho_1,\varphi_1,\ldots,e^{t/2}\rho_{n-2},\varphi_{n-2}).$$
Consider the 1--forms
$$\widetilde\lambda_{ex}=e^t\alpha_{ot}+\sum_{i=1}^{n-2}\rho^2_id\varphi_i,\quad \lambda_{ex}=e^t\alpha_{ex}=e^t(\alpha_{ot}+\sum_{i=1}^{n-2}\rho^2_id\varphi_i)$$
The diffeomorphism $f$ pulls--back $f^*\widetilde\lambda_{ex}=\lambda_{ex}$. In particular
$$(\SS(\R^3,\xi_{ot})\times\R^{2n-4},\omega_{ot}+\omega_0)\cong\SS(\R^{2n-1},\alpha_{ex})$$
are symplectomorphic. This concludes the statement.\hfill$\Box$\\

Proposition \ref{prop:higher} can also be used to prove an analogue of Theorem \ref{thm:openmfld} in higher dimensions.

\end{document}